\documentclass[11pt,a4paper]{article}

\usepackage[T1]{fontenc}
\usepackage{lmodern}
\usepackage[margin=1in]{geometry}
\usepackage{amsmath,amssymb,amsthm}
\usepackage{booktabs}
\usepackage{microtype}
\usepackage[colorlinks=true,linkcolor=blue,citecolor=blue,urlcolor=blue]{hyperref}
\hypersetup{pdftitle={A Peierls bound for planar soft-stick percolation},pdfauthor={Shitao Chen}}

\newtheorem{theorem}{Theorem}
\newtheorem{lemma}[theorem]{Lemma}
\newcommand{\Z}{\mathbb Z}
\newcommand{\PP}{\mathbb P}
\newcommand{\cl}{\mathcal C}
\newcommand{\tr}{\operatorname{tr}}
\newcommand{\eps}{\varepsilon}
\title{A Peierls bound for planar soft-stick percolation}
\author{Shitao Chen}
\date{}

\begin{document}
\maketitle

\begin{abstract}
In planar soft-stick percolation, each vertex of the square lattice independently
opens one uniformly chosen outgoing arrow and, with probability $\eps$, the
opposite arrow as well. Motivated by the question on its critical parameter
raised by B\"aumler et al., we prove that $\eps_c\leq0.99<1$, establishing
percolation below the two-arrow endpoint. More precisely, at $\eps=0.99$
the origin has an infinite forward cluster with probability greater than
$11/20$. The proof uses a Peierls argument in which boundary turns are
encoded by a three-state transfer matrix to bound the contour sum.
\end{abstract}

\section{Introduction and main result}

The soft-stick model interpolates between a one-arrow environment and a
lattice of horizontal and vertical two-arrow sticks. Percolation at the
two-arrow endpoint is known. To study parameters below this endpoint, we use
the geometry of the outer boundary of a finite forward cluster, accounting
for dependencies between arrows leaving the same vertex.

Independently for each $x\in\Z^2$, choose $A_x\in\{H,V\}$ uniformly.
Conditional on $A_x$, open both arrows on that axis with probability
$\eps\in[0,1]$. With probability $\delta:=1-\eps$, open just one of them,
uniformly. Let $\cl(0)$ be the set of vertices reachable from the origin by
directed open paths, and set
\begin{equation}\label{eq:definitions}
\theta(\eps):=\PP_\eps\bigl(|\cl(0)|=\infty\bigr),
\qquad
\eps_c:=\sup\{\eps\in[0,1]:\theta(\eps)=0\}.
\end{equation}

\begin{theorem}\label{thm:main}
For planar soft-stick percolation,
\[
\theta\left(\frac{99}{100}\right)>\frac{11}{20}.
\]
In particular, $\eps_c\leq 0.99<1$.
\end{theorem}

The model is the soft regime of the northsouth-eastwest model in Coupier,
Henry, Jahnel and K\"oppl~\cite{CHJK}, with parameter $p=(1+\eps)/4$.
Their Lemma~2.7 states $\theta_{1/2}^{\mathrm{ns\text{-}ew}}=1$, which is
$\theta(1)=1$ in our parameterization.
B\"aumler et al.~\cite[Section~3.3]{BJKLRT} ask about the critical soft-stick
parameter and explain that the comparison considered there does not yield
a percolation conclusion for this model. Theorem~\ref{thm:main} gives an
explicit upper bound for this critical parameter, showing that percolation
occurs below the two-arrow endpoint.

Coupier et al.~\cite{CHJK} also prove strict critical-parameter bounds for
the directed-corner and unrestricted two-neighbor models, whose local arrow
distributions differ from that considered here. Their numerical study
in~\cite[Section~2.3.4]{CHJK} suggests $p_c^{\mathrm{ns\text{-}ew}}$ around
$0.42$, corresponding to $\eps_c$ around $0.68$.

\newpage
\section{Closed boundaries and local probabilities}

We first record monotonicity. At each vertex, sample an axis, a preferred
direction on that axis, and an independent uniform variable $U_x\in[0,1]$.
At parameter $t$, keep the preferred arrow open and add the opposite arrow
when $U_x\leq t$. This coupling has the required law and only adds arrows as
$t$ increases. Thus $\theta$ is nondecreasing; once the probability in
Theorem~\ref{thm:main} is positive, its assertion about $\eps_c$ follows.

For a nonempty set $R_x$ of directions, consider the event that every outgoing
arrow in $R_x$ is closed. Write $r=1/2+\delta/4$. The six possible states at a
vertex give the following probabilities.

\begin{table}[ht]
\centering
\renewcommand{\arraystretch}{1.2}
\begin{tabular}{lccccc}
\toprule
Requirement & One & Two opposite & Two adjacent & Three & Four\\
\midrule
Probability & $r$ & $1/2$ & $\delta/2$ & $\delta/4$ & $0$\\
\bottomrule
\end{tabular}
\caption{Exact costs of closing a prescribed set of outgoing arrows.}
\label{tab:costs}
\end{table}

Indeed, each one-arrow state has probability $\delta/4$, and each two-arrow
stick has probability $(1-\delta)/2$. A prescribed arrow is closed if the
perpendicular axis is chosen, or if the opposite one-arrow state is chosen.
Two opposite arrows are closed precisely when the perpendicular axis is
chosen. Two adjacent arrows leave two admissible one-arrow states, while
three prescribed arrows leave only one. Closing all four is impossible.
Requirements at different vertices are independent.

Let $(\Z^2)^*=\Z^2+(1/2,1/2)$ be the dual lattice. For a finite, nonempty
nearest-neighbor connected set $S\subset\Z^2$, let $U_\infty(S)$ be the
infinite component of $\Z^2\setminus S$, and let
$\widehat S=\Z^2\setminus U_\infty(S)$ be its filled hull. Its external edge
boundary is
\[
B(S)=\bigl\{\{x,y\}:x\in\widehat S,\ y\in U_\infty(S),\ |x-y|_1=1\bigr\}.
\]

\begin{lemma}[Outer boundary]\label{lem:boundary}
The hull $\widehat S$ is finite and connected. Every edge of $B(S)$ has its
hull endpoint in $S$, and the dual edges $B(S)^*$ form one simple circuit
surrounding $\widehat S$.
\end{lemma}

\begin{proof}
Choose a box containing $S$ in its interior. All vertices outside the box
belong to one infinite complementary component, so this component is unique
and the hull is finite. Every finite complementary component has a
nearest-neighbor edge to $S$; adjoining these components preserves
connectedness. A vertex of $\widehat S\setminus S$ cannot be adjacent to
$U_\infty(S)$, since both would then belong to the same component of
$\Z^2\setminus S$.

Both $\widehat S$ and $U_\infty(S)$ are connected. Therefore $B(S)$ is a
minimal nonempty edge cut: restoring any one of its edges reconnects the
lattice. At every dual vertex, the number of incident dual boundary edges
is even, as membership in the hull changes an even number of times around
a primal square. Hence the finite dual boundary contains a simple cycle
$\gamma$.

The primal edges crossed by $\gamma$ form a nonempty edge cut by the Jordan
curve theorem. They are a subset of $B(S)$, so minimality forces equality.
Thus the whole dual boundary is the simple cycle $\gamma$. The connected
infinite complementary set lies outside it. Every crossed primal edge has
endpoints on opposite sides, so the connected hull lies inside it.
\end{proof}

\section{Turning geometry and contour weights}

Orient every simple dual circuit $\Gamma$ counterclockwise, with its bounded
side on the left. Let $\ell$ be its length, and let $c$ and $k$ be its numbers
of left and right turns. We call them convex and concave turns. The signed
turning angle gives
\begin{equation}\label{eq:turning}
c-k=4.
\end{equation}

For an inside primal vertex $x$, let $R_x(\Gamma)$ contain the directions of
the primal edges leaving $x$ and crossed by $\Gamma$. Let $E_\Gamma$ be the
event that all these arrows are closed. Classify the nonempty requirements
by their shapes: one direction, two opposite directions, two adjacent
directions, or three directions. Denote the corresponding vertex counts by
$n_1,n_o,n_a,n_3$.

\begin{lemma}[Boundary bookkeeping]\label{lem:bookkeeping}
If $\PP_\eps(E_\Gamma)>0$, then
\begin{equation}\label{eq:bookkeeping}
\ell=n_1+2n_o+2n_a+3n_3,
\qquad c=n_a+2n_3.
\end{equation}
The indexed cyclic turn word of $\Gamma$, with letters $L,S,R$ for left,
straight and right, contains no cyclic occurrence of $LLL$. If $m$ is its
number of cyclic occurrences of $LL$, then $m=n_3$.
\end{lemma}

\begin{proof}
The dual edges crossing the primal edges incident to $x$ are the four sides
of the dual unit square centered at $x$. Partitioning contour edges by their
inside primal endpoints gives the first identity; the zero cost of four
closed arrows excludes that case.

At a convex turn, the two edges are adjacent sides of one such square and
share its center as their inside endpoint. Conversely, two adjacent
requirements use two consecutive sides of that square, and three
requirements use three consecutive sides. They contribute one and two
convex turns, respectively. The remaining requirement types contribute
none, proving the second identity.

Two consecutive left turns traverse three sides of one dual unit square and
hence correspond exactly to a three-arrow requirement. Three consecutive
left turns would complete that square and return to the initial dual vertex.
Simplicity then forces the entire circuit to be this square, whose boundary
event has probability zero. This case was excluded.
\end{proof}

Put $\alpha=1/\sqrt2$ and assume $0<\delta\leq1$ and $r\leq\alpha$.
For contours with positive boundary-event probability, independence and
the local costs give
\begin{equation}\label{eq:weight}
\PP_\eps(E_\Gamma)
=r^{n_1}\left(\frac12\right)^{n_o}
 \left(\frac\delta2\right)^{n_a}\left(\frac\delta4\right)^{n_3}
\leq\alpha^{\ell+m}\delta^{c-m}.
\end{equation}
Using $1/2=\alpha^2$ and $1/4=\alpha^4$, we bound the product by
$\alpha^{n_1+2n_o+2n_a+4n_3}\delta^{n_a+n_3}$.
Lemma~\ref{lem:bookkeeping} now gives the stated bound. The inequality also
holds for zero-probability contours.

Fix $z>0$. Since $c-k=4$, we may rewrite this bound as
\begin{equation}\label{eq:tilt}
\PP_\eps(E_\Gamma)\leq z^{-4}W(\Gamma),
\qquad W(\Gamma)=z^{c-k}\alpha^{\ell+m}\delta^{c-m}.
\end{equation}
The parameter $z$ allows us to sum over an enlarged collection of turn
words without imposing the constraint $c-k=4$ on every word. Three states
keep track of consecutive left turns.

\section{A three-state contour sum}

For an indexed cyclic word $w\in\{L,S,R\}^\ell$, let $c(w)$ and $k(w)$
count its left and right turns, and let $m(w)$ count cyclic occurrences
of $LL$. Give it the weight
\[
W(w)=z^{c(w)-k(w)}\alpha^{\ell+m(w)}\delta^{c(w)-m(w)}.
\]
We sum over words with no cyclic occurrence of $LLL$. The word positions
remain indexed: words are not identified under cyclic shifts.

Use states $0,1,2$ for the length of the current terminal run of left turns.
State $0$ means that the current letter is not $L$. Introduce the matrix
\begin{equation}\label{eq:matrix}
T=T_{\delta,z}=
\begin{pmatrix}
\alpha(1+z^{-1})&\alpha\delta z&0\\
\alpha(1+z^{-1})&0&\alpha^2z\\
\alpha(1+z^{-1})&0&0
\end{pmatrix}.
\end{equation}

\begin{lemma}[Cyclic word sum]\label{lem:words}
For every $\ell\geq4$,
\begin{equation}\label{eq:trace}
\sum_{\substack{w\in\{L,S,R\}^{\ell}\\
                 w\text{ has no cyclic }LLL}}W(w)=\tr(T^\ell).
\end{equation}
\end{lemma}

\begin{proof}
Every admissible word contains a non-$L$ letter, so it determines a unique
cyclic state sequence. A transition into state $0$ represents $S$, with
weight $\alpha$, or $R$, with weight $\alpha/z$. The transition $0\to1$
is a first left turn, with weight $\alpha\delta z$. The transition $1\to2$
is a second consecutive left turn. It receives weight $\alpha^2z$, which
supplies the extra factor $\alpha/\delta$ for one occurrence of $LL$.

These are exactly the entries of $T$. Expanding $\tr(T^\ell)$ counts closed
state walks with indexed positions. Expanding each factor
$\alpha(1+z^{-1})$ distinguishes $S$ from $R$. The resulting products are
precisely $W(w)$, with no omission or multiplicity.
\end{proof}

\begin{lemma}[Rooted contour bound]\label{lem:root}
If $r\leq\alpha$, then for every $\ell\geq4$,
\begin{equation}\label{eq:root}
\sum_{\substack{\Gamma\text{ simple dual circuit}\\
                 |\Gamma|=\ell,\ 0\text{ inside }\Gamma}}
\PP_\eps(E_\Gamma)
\leq\frac{\ell}{2z^4}\tr(T^\ell).
\end{equation}
\end{lemma}

\begin{proof}
Discard circuits whose boundary event has probability zero. Root each
remaining circuit at its nearest intersection with the positive horizontal
ray. This lies in a vertical dual edge at a positive half-integer coordinate
$u$. The segment from the origin to that intersection is inside the circuit,
so the counterclockwise orientation traverses the root edge upwards.

The circuit also meets the negative horizontal ray at some $-v<0$. Each of
the two arcs between these intersections has horizontal variation at least
$u+v$. Thus $\ell\geq2(u+v)>2u$, leaving at most $\ell/2$ possible root
edges. A fixed root edge and an indexed turn word determine the circuit
uniquely. By Lemma~\ref{lem:bookkeeping} its word has no $LLL$, and its
probability is at most $z^{-4}W(w)$. Summing over the possible roots and then
over all admissible words proves the bound. Words which do not form closed
simple circuits only enlarge the sum.
\end{proof}

\section{An explicit estimate at \texorpdfstring{$\eps=0.99$}{epsilon=0.99}}

Take $\delta=1/100$ and $z=11/2$. Here $r=201/400<\alpha$, since
$201^2<80000$. Direct expansion of the three-by-three determinant gives
\begin{equation}\label{eq:det}
P(t):=\det(I-tT)=1-b_1t-b_2t^2-b_3t^3,
\end{equation}
where
\begin{equation}\label{eq:coefficients}
b_1=\frac{13}{11\sqrt2},\qquad
b_2=\frac{13}{400},\qquad
b_3=\frac{143}{1600}.
\end{equation}
The inequality $3250^2<2\cdot2299^2$ gives $b_1<209/250$. Hence, writing
$d=1-b_1-b_2-b_3$, we have
\begin{equation}\label{eq:rational-bounds}
d>\frac{337}{8000}>\frac{21}{500},\qquad
b_1+4b_2+9b_3<\frac95,\qquad
b_1+2b_2+3b_3<\frac{117}{100}.
\end{equation}
For the last two comparisons, the respective upper bounds obtained by
substituting $209/250$ are $14163/8000$ and $9353/8000$.

Since $T$ is nonnegative, its spectral radius is a nonnegative eigenvalue.
Its characteristic polynomial is $\lambda^3-b_1\lambda^2-b_2\lambda-b_3$.
For $\lambda\geq1$ this polynomial is at least $\lambda^3d>0$, so
$\rho(T)<1$. The trace series and its derivative therefore converge in a
neighborhood of $t=1$. Logarithmic differentiation of $P$ gives
\[
-\frac{tP'(t)}{P(t)}=\sum_{\ell\geq1}\tr(T^\ell)t^\ell.
\]
Differentiating once more and setting $t=1$, we obtain the exact identity
\begin{equation}\label{eq:sum}
\begin{aligned}
\sum_{\ell\geq1}\ell\tr(T^\ell)
&=\frac{b_1+4b_2+9b_3}{d}
  +\frac{(b_1+2b_2+3b_3)^2}{d^2}\\
&<\frac{300}{7}+\left(\frac{195}{7}\right)^2
 =\frac{40125}{49}<820.
\end{aligned}
\end{equation}

\begin{proof}[Proof of Theorem~\ref{thm:main}]
If $\cl(0)$ is finite, it is nearest-neighbor connected and every outgoing
arrow to its complement is closed. Lemma~\ref{lem:boundary} gives a simple
outer dual circuit around the origin, with each crossed edge having its
inside endpoint in $\cl(0)$. Thus its boundary event occurs. A union bound,
Lemma~\ref{lem:root} and the preceding estimate give
\[
\PP_{99/100}\bigl(|\cl(0)|<\infty\bigr)
\leq\frac{1}{2z^4}\sum_{\ell\geq4}\ell\tr(T^\ell)
<\frac{8\cdot820}{14641}<\frac9{20}.
\]
We enlarged the sum to all $\ell\geq1$, and used $1/(2z^4)=8/14641$.
The final strict inequality follows from $131200<131769$. Therefore
$\theta(99/100)>11/20$. Monotonicity yields $\eps_c\leq99/100<1$.
\end{proof}

In the parameterization of~\cite{CHJK}, this reads
$p_c^{\mathrm{ns\text{-}ew}}\leq(1+0.99)/4=0.4975<1/2$.
This bound establishes percolation below the endpoint; no sharpness is claimed.

\paragraph{Acknowledgments.}
The author thanks Professor Hao Ge for his guidance and helpful comments on
this paper. Based on the author's ideas for the preceding estimates, GPT
assisted in selecting the parameter $z=11/2$. GPT also assisted with language
editing.


\begin{thebibliography}{9}

\bibitem{CHJK}
D.~Coupier, B.~Henry, B.~Jahnel and J.~K\"oppl.
\newblock \emph{The planar lattice two-neighbor graph percolates}.
\newblock arXiv:2412.20781 (2024).
\newblock \url{https://arxiv.org/abs/2412.20781}.

\bibitem{BJKLRT}
J.~B\"aumler, B.~Jahnel, J.~K\"oppl, B.~Lodewijks, L.~Reeves and A.~T\'obi\'as.
\newblock \emph{Local criteria for global connectivity comparisons:
beyond stochastic domination}.
\newblock arXiv:2510.03934v2 (2026).
\newblock \url{https://arxiv.org/abs/2510.03934v2}.

\end{thebibliography}
\end{document}